\documentclass[hidelinks]{amsart}
\usepackage{amsfonts}
\usepackage{graphicx}
\usepackage{amsmath}
\usepackage{amssymb}
\usepackage{amscd}
\usepackage{enumitem}
\usepackage{amsthm}
\usepackage[mathscr]{euscript}
\usepackage{mathtools}
\usepackage{tikz-cd}
\usepackage{macros}
\usepackage{hyperref}
\usepackage{tikz}
\usepackage{tabularx}
\usepackage{multirow}

\newcommand{\p}[1]{\medskip\noindent\textit{#1}.}

\title{Geodesic-preserving bijections of the Thurston geometries}
\author{Ryan Dickmann}
\address{Vanderbilt University}
\email{ryan.dickmann@vanderbilt.edu} 
\author{Palani Lideros} 
\address{Ohio State University}
\email{lideros.1@osu.edu}
\author{Akash Narayanan}
\address{University of Notre Dame}
\email{anaraya2@nd.edu}

\begin{document}

\begin{abstract}
    We completely classify the bijections of the Thurston geometries that preserve geodesics as sets. For Riemannian manifolds that satisfy a certain technical condition, we prove that a totally geodesic subset is a submanifold. We also classify the geodesic-preserving bijections of the Euclidean cylinder $\Sp^1 \times \R$ and the bijections of the hyperbolic plane $\Hyp^2$ that preserve constant curvature curves.
\end{abstract}
\maketitle
\section{Introduction}
Given a complete Riemannian manifold $M$, we say a bijection $f: M \rightarrow M$ is \textit{geodesic-preserving} if the image of each geodesic is a geodesic as a set. Throughout the paper, a geodesic refers to the image of a locally isometric immersion of the real line into a Riemannian manifold; in other words, we always assume that geodesics are complete. Note that $f$ is not assumed to be continuous.

In his paper ``Lost Theorems of Geometry" \cite{Jeffers}, Jeffers classifies the geodesic-preserving bijections for Euclidean spaces $\E^n$, spherical spaces $\Sp^n$, and hyperbolic spaces $\Hyp^n$ for all $n$. This completely classifies geodesic-preserving bijections for the possible two-dimensional geometries given by the uniformization theorem. 

In three dimensions, there are eight geometries, referred to as the Thurston geometries due to Thurston's work classifying them. By Jeffers, geodesic-preserving bijections are understood for three of these geometries, $\Hyp^3, \E^3,$ and $\Sp^3$, so we investigate the remaining five geometries:
\[
\begin{array}{c}
\HxR \quad\quad \SxR \quad\quad \sltilde \quad\quad \Nil \quad\quad \Sol
\end{array}
\]

An isometry is always geodesic-preserving; in addition, for $\HxR$ and $\SxR$, any map that acts as the identity on the first coordinate and an affine map in the $\R$-coordinate is geodesic-preserving. We classify the geodesic-preserving bijections of these five Thurston geometries and show that in each case these basic examples generate the entire group of geodesic-preserving bijections. 

  \begin{theorema} \label{thm}
  Let $f$ be a geodesic-preserving bijection of $\Hyp^2 \times \R$ or $\Sp^2 \times \R$. Then $f$ is an isometry composed with an affine map in the $\R$-coordinate. Let $f$ be a geodesic-preserving bijection of $\sltilde$, $\Nil$, or $\Sol$. Then $f$ is an isometry.
 \end{theorema}

Here $\sltilde$ denotes the universal cover of the Lie group $\slg$. We equip $\slg$ with the left-invariant Riemannian metric with respect to the Lie group structure, and we equip $\sltilde$ with the pullback metric.

$\Nil$ geometry is defined by the left-invariant Riemannian metric on the Heisenberg group, which has the  following group operation on $\R^3$: \[(x,y,z) \cdot (a,b,c) = (x+a,y+b,z+c+xb)\]

$\Sol$ geometry is defined similarly using the Lie group given by the following operation on $\R^3$: \[(x, y, z) \cdot (a, b, c) = (e^{-z}a + x, e^{z}b + y, c + z)\]

 \subsection*{Auxiliary results} The proof of Theorem \hyperref[thm]{A} relies on the following results.
 
  \p{Totally geodesic subsets} A \textit{totally geodesic} subset $X$ is a nonempty subset such that for any $p,q \in X$ there exists a geodesic between $p$ and $q$ that lies entirely in $X$. We say a totally geodesic subset is \textit{trivial} when it is either a geodesic or the entire space. Totally geodesic subsets are particularly useful since a geodesic-preserving bijection restricts to a geodesic-preserving bijection between totally geodesic subsets. 

Note a totally geodesic subset is not a priori given any topological structure, such as an embedded or immersed submanifold, similar to how a geodesic-preserving bijection is not assumed to be continuous. We make use of the following result for Riemannian manifolds where geodesics never return too close to themselves, and we prove it in the next section.
 
  \begin{proposition}[Totally geodesic subsets are submanifolds] \label{prop: fundlemma} 
    Let $M$ be a Riemannian manifold with the property that there is an $\varepsilon > 0$ such that the intersection of a geodesic and an open $\varepsilon$-ball is either empty or a single geodesic segment. Then any totally geodesic subset is an embedded submanifold.
  \end{proposition}

 \p{Cylinder case} For the proof of the $\Sp^2 \times \R$ case, we first address the case of the Euclidean cylinder $\Sp^1 \times \R$ since cylinders appear as totally geodesic subsurfaces in $\Sp^2 \times \R$. We classify the geodesic-preserving bijections of the cylinder in Theorem~\ref{prop: cylinder2}. 
 
 \p{Constant curvature curves} For the proof of the $\sltilde$ case, we classify the bijections of $\Hyp^2$ that preserve constant curvature curves in Proposition~\ref{constantcurvature}. This is useful to us due to the connection between $\sltilde$ and $\Hyp^2$. In particular, $\sltilde$ is a line bundle over $\Hyp^2$, and certain geodesics project to curves of constant curvature.

  \subsection*{Outline} The proof of Theorem \hyperref[thm]{A} is split into multiple sections with each of the five geometries given its own section. We outline the main ideas used in each case. We prove Proposition~\ref{prop: fundlemma} in Section~\ref{section:proof}. This condition holds for all five considered Thurston geometries except $\Sp^2 \times \R$, and it is used during the proofs of the $\Hyp^2 \times \R$ and $\Sol$ cases since these have interesting totally geodesic subsets. The $\sltilde$ and $\Nil$ geometries are known to have no totally geodesic submanifolds and thus no totally geodesic subsets.

 \p{$\HxR$ case} We prove this case in Section~\ref{sec: hyp} with Theorem~\ref{hyp}. We give an elementary proof for the classification of the totally geodesic submanifolds, then use Proposition~\ref{prop: fundlemma} to classify the totally geodesic subsets. The nontrivial totally geodesic subsets are exactly the horizontal planes (isometric to $\Hyp^2$) and the vertical planes (isometric to $\E^2$). We show these are each preserved by a geodesic-preserving bijection, then combine with the known classification of geodesic-preserving bijections for $\Hyp^2$ and $\E^2$ by Jeffers. 

  \p{$\SxR$ case} In Section~\ref{sec: sphere}, we first prove the cylinder case with Theorem~\ref{prop: cylinder2}, and then the $\SxR$ case with Theorem~\ref{sphere}. The nontrivial totally geodesic submanifolds of $\SxR$ are horizontal spheres (isometric to $\Sp^2$) and vertical cylinders (isometric to $\Sp^1 \times \R$). We directly show spheres and cylinders are preserved by a geodesic-preserving bijection by using properties of the geodesics in $\SxR$. We then follow a similar method to the $\Hyp^2 \times \R$ case and appeal to the known classification of geodesic-preserving bijections for $\Sp^2$ and the cylinder case.
  
  \p{$\sltilde$ case} We prove this case in Section~\ref{sec: sltilde} with Theorem~\ref{sl}. Our proof uses that this geometry is a line bundle over $\Hyp^2$ and the projection of a geodesic to $\Hyp^2$ is either a point or a constant curvature curve. We show that a geodesic-preserving map of $\sltilde$ preserves this line bundle and gives a well-defined bijection on $\Hyp^2$ that preserves the set of constant curvature curves, and then we appeal to Proposition~\ref{constantcurvature} which shows a bijection preserving constant curvature curves is an isometry.

  \p{$\Nil$ case} We prove this case in Section~\ref{sec: nil} with Theorem~\ref{nil}. The proof is similar to the $\sltilde$ case using the fact that $\Nil$ is a line bundle over $\E^2$ and the projection of a geodesic is either a point or a constant curvature curve (circles or lines). We show that a geodesic-preserving bijection $f$ of $\Nil$ gives a well-defined bijection $f_\star$ on $\E^2$, and then show directly $f_\star$ takes lines to lines. We then appeal to the classification of geodesic-preserving bijections of $\E^2$.

    \p{$\Sol$ case} We prove this case in Section~\ref{sec: sol} with Theorem~\ref{sol}. Through each point in $\Sol$ there are two orthogonal hyperbolic planes, and these are exactly the nontrivial totally geodesic subsurfaces of $\Sol$. We assume this as fact, and use this with Proposition~\ref{prop: fundlemma} to classify the totally geodesic subsets. Then a geodesic-preserving bijection of $\Sol$ takes hyperbolic planes to hyperbolic planes, and we appeal to the classification of geodesic-preserving bijections of $\Hyp^2$.

  \subsection*{Related problems}

  We pose some questions and discuss related work in the area. We first wonder if the condition in Proposition~\ref{prop: fundlemma} can be removed. 
  
 \begin{question} \label{question}
 For a general Riemannian manifold is a totally geodesic subset always an immersed submanifold?
 \end{question}
 
 \begin{remarknum}
     It may be useful for Question~\ref{question} to use the fact that any two geodesics in a Riemannian manifold either agree or intersect a countable number of times. We could not find this fact stated in the literature, but it is an exercise using the Riemannian exponential map and separability of the manifold. 

     The Baire Category Theorem may also be useful to show a totally geodesic subset contains a ``fan" of geodesics dense in an open neighborhood of the space of geodesics through a point. Compare with the proof of Proposition~\ref{prop: fundlemma} in the next section.
 \end{remarknum}

Totally geodesic subsets that are not a priori submanifolds were studied for complex hyperbolic space \cite{botos2024}. Although Proposition~\ref{prop: fundlemma} lets us classify totally geodesic subsets for some of the Thurston geometries, it may be possible to directly classify totally geodesic subsets for other Riemannian manifolds where this proposition does not apply.

 \begin{problem}
     Classify totally geodesic subsets for other Riemannian manifolds.
 \end{problem}

 %\paragraph{Related questions}

  There are countless variations to the geodesic-preserving bijection problem. Recently, there has been progress classifying geodesic-preserving bijections for compact manifolds. For example, see \cite{ShulkinVanLimbeek2017}, \cite{Yang2021}, \cite{Yang2025}. Notably, the case of compact hyperbolic surfaces without boundary is still open. We pose the following similar problem.

  \begin{problem}
  Classify geodesic-preserving bijections for a compact quotient of a Thurston geometry.
  \end{problem}

A variation to this type of problem involves replacing geodesics with some other geometric object for which incidence axioms are able to be defined. The most classical result of this type is due to Carathéodory who showed a circle-preserving bijection between any open connected subsets of $\R^2$ is the restriction of a M\"{o}bius transformation \cite{caratheodory1937}. Lo and Sane \cite{karim2024} proved that a bijection of hyperbolic space that either preserves horocycles or preserves hypercycles is an isometry. This is similar to Theorem~\ref{prop: cylinder2} except we consider bijections that may change the type of the constant curvature curve.

Yet another variation involves relaxing some of the given conditions for any of the related problems, for example, replacing bijection with injection or surjection. See \cite{Artstein-Avidan2016}, \cite{ChubarevPinelis1999}, \cite{GibbonsWebb1979}, \cite{LiWang2016}, \cite{yao2011}. See \cite{Ozgur2014} for a detailed report on the circle-preserving case.

We say a function between Riemannian manifolds is \textit{geodesic} when the image of each geodesic is contained in a geodesic. An important example is given by the Klein model of hyperbolic space as the open unit disk in Euclidean space of the same dimension. With this example in mind, we pose the following problems.

\begin{problem}
       Show every non-constant geodesic function from $\Hyp^n$ to $\E^n$ is given by the Klein model on some open disk.
\end{problem}

\begin{problem}
Classify the non-constant geodesic functions between any of the geometries.
\end{problem}

\subsection*{Acknowledgements}

This project began during the 2023 CUBE REU. We thank Dan Margalit for suggesting the classification of geodesic-preserving bijections of the Thurston geometries and for comments on an earlier draft. Further thanks to the other mentors/organizers Sahana Balasubramanya, Wade Bloomquist, Abdoul Karim Sane, and Roberta Shapiro.

\section{Totally geodesic subsets are submanifolds} \label{section:proof}

We prove Proposition~\ref{prop: fundlemma} through an induction argument in which the condition on the manifold is used to build an open ball of dimension $n$ in the totally geodesic subset from a given open ball of dimension $n-1$ in the totally geodesic subset.
 
  \begin{proof} [Proof of Proposition~\ref{prop: fundlemma}]
  Suppose $M$ is dimension $n$. Let $X$ be a totally geodesic subset. Fix a point $x$ and consider all the geodesics in $X$ through $x$. Let $Y$ be the pullback of $X$ via the exponential map at $x$ to $\R^n$. Let $Y^\prime$ be the linear subspace of $\R^n$ spanned by the vectors of $Y$ through the origin, and let the dimension of $Y^\prime$ be $m \leq n$. It suffices to show $Y = Y^\prime$. Note it suffices to show by induction that a $j$-dimensional subspace of $Y^\prime$ is contained in $Y$ for all $j\leq m$. 
  
  The $j=1$ case is trivial. For general $j$, choose any $(j-1)$-dimensional subspace of $Y^\prime$. By considering the image of this subspace and using the induction hypothesis, $X$ contains an $(j-1)$-ball. Let $\gamma$ be a small geodesic segment through $x$ intersecting the ball only at $x$. This is possible due to the local diffeomorphism property of the exponential map. 
  
  By choosing a point $y \in \gamma \setminus\{x\}$ sufficiently close, it follows from the condition on $M$ and properties of the exponential map that there is a unique geodesic from $y$ to each point of the ball. Since $X$ is totally geodesic, these geodesics are in $X$, and it follows $X$ contains an $j$-ball centered about $x$. Using the condition on $M$ again, we see $X$ contains geodesics from $x$ to nearby points of the $j$-ball, and so the desired $j$-dimensional subspace of $Y^\prime$ is contained in $Y$.
  \end{proof}

\section{\texorpdfstring{Geodesic-preserving bijections of $\HxR$}{Geodesic-preserving bijections of H2 times R}} \label{sec: hyp}

We want to show that a geodesic-preserving bijection of $\HxR$ is an isometry composed with an affine map in the $\R$-coordinate. The isometry group is exactly $\Isom(\Hyp^2) \times \Isom(\R)$, so we may then describe the group of geodesic-preserving bijections as $\Isom(\Hyp^2) \times \Aff(\R) $.

\begin{theorem} \label{hyp}
    Let $f: \HxR \rightarrow \HxR$ be a geodesic-preserving bijection. Then $f$ is an isometry composed with an affine map in the $\R$-coordinate.
\end{theorem}
\p{Geodesics} For Riemannian product metrics, a curve is a geodesic of the product space if and only if it can be parameterized so the projection to each coordinate gives a geodesic parameterized at constant speed or a point. For the $\HxR$ and $\SxR$ cases, we say a geodesic is \textit{vertical} if it is equal to $\{p\} \times \R$ for some point $p$ in the two-dimensional factor, and we say it is \textit{horizontal} if it is contained in $\Sp^2 \times \{r\}$ or $\Hyp^2 \times \{r\}$ for some $r \in \R$. There is one other type of geodesic we call a \textit{slant} geodesic such that the projection to either factor is a geodesic.

There are two types of natural totally geodesic subsets in $\HxR$ -- \textit{horizontal planes} of the form $\Hyp^2 \times \{r\}$ and \textit{vertical planes} are of the form $\gamma \times \R$ for some geodesic $\gamma \in \Hyp^2$. Note that horizontal planes are isometric to $\Hyp^2$ and vertical planes are isometric to $\E^2$. 
The totally geodesic subsurfaces of $\HxR$ are well-known to be exactly the horizontal and vertical planes.

From the above description of the geodesics of $\HxR$, Proposition~\ref{prop: fundlemma} applies, so that a totally geodesic subset of $\HxR$ is an embedded submanifold. Thus, the nontrivial totally geodesic subsets are exactly the horizontal and vertical planes. We provide an elementary proof that relies only on the classification of geodesics in $\HxR$ and Proposition~\ref{prop: fundlemma}.

\begin{classh}
A nontrivial totally geodesic subset in $\Hyp^2 \times \R$ is either a horizontal or vertical plane.
\end{classh}

\begin{proof}

Consider some nontrivial totally geodesic subset $S \subset \HxR$ which by Proposition~\ref{prop: fundlemma} is an embedded subsurface. Consider the intersection of $S$ with each horizontal plane. If $S$ only intersects a single horizontal plane, then it follows that $S$ must be the entire horizontal plane. Otherwise, $S$ must intersect every horizontal plane since it will contain a slant or vertical geodesic. The intersection of $S$ with each horizontal plane must be a single geodesic, since otherwise $S$ would contain an entire horizontal plane and then the entirety of $\HxR$. It follows $S$ is homeomorphic to $\R^2$.

 \begin{figure}[ht]
    \centering
    {    \includegraphics[width=0.45\textwidth]{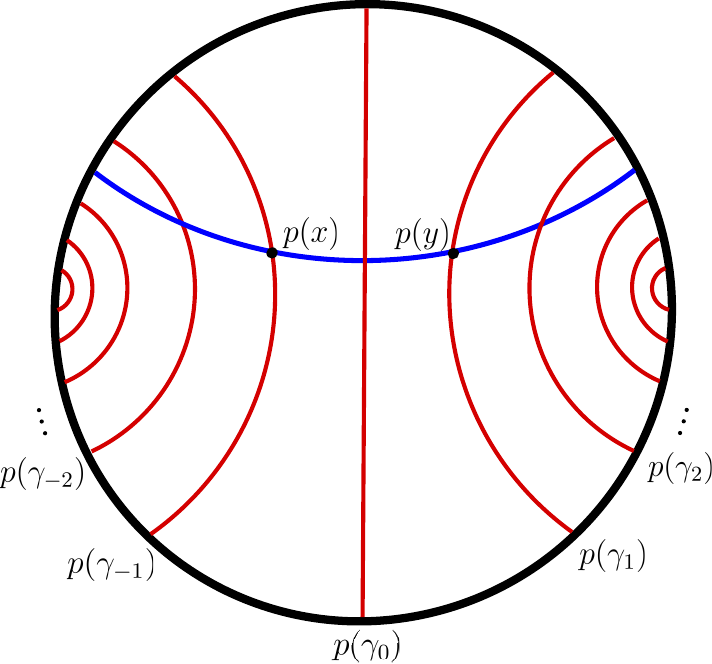} % Adjust width as needed
   \caption{For classifying the totally geodesic subsets of $\HxR$, the chosen geodesic from $p(x)$ to $p(y)$ intersects $p(\gamma_r)$ only for $r$ in a bounded range. The $p(\gamma_r)$ for integer $r$ are shown.}
    \label{figure:hyp}
    }
\end{figure}

 Let $\gamma_r$ denote the geodesic $S \cap (\Hyp^2 \times \{r\})$, and let $p$ denote the projection map to $\Hyp^2$. Observe if any two $p(\gamma_r)$ intersect, then $S$ contains a vertical geodesic. This vertical geodesic along with one of the $\gamma_r$ forces $S$ to contain a vertical plane, so we can assume otherwise that the $p(\gamma_r)$ are disjoint.

 Note that the $p(\gamma_r)$ are ordered monotonically in $\Hyp^2$, meaning $p(r_1)$ is to the left (without loss of generality) of $p(r_2)$ whenever $r_1 < r_2$. To see this, consider a slant geodesic in $S$ and note it hits every horizontal plane in order and thus every $\gamma_r$ in order.
 
Let $x$ and $y$ be arbitrary points on distinct $\gamma_r$, and call the slant geodesic between them $\gamma_{xy}$. We have that $\gamma_{xy}$ is contained in $S$ since it is totally geodesic. Note  $p(\gamma_{xy})$ is the geodesic of $\Hyp^2$ from $p(x)$ to $p(y)$. By choosing $x$ and $y$ appropriately, we can ensure the geodesic from $p(x)$ to $p(y)$ intersects $p(\gamma_r)$ only for $r$ within a bounded range. See Figure~\ref{figure:hyp}. This is a contradiction, since $\gamma_{xy}$ is a slant geodesic that must intersect each horizontal plane and thus each $\gamma_r$. \end{proof}

From this we obtain additional structure to work with for a geodesic-preserving bijection.

\begin{lemma} \label{planes}
    A geodesic-preserving bijection $f: \HxR \rightarrow \HxR$ takes horizontal planes to horizontal planes and vertical planes to vertical planes. Furthermore, $f$ restricts to an isometry on horizontal planes and an affine map on vertical planes.
\end{lemma}

\begin{proof} 
Recall horizontal planes are isometric to $\Hyp^2$ and vertical planes are isometric to $\E^2$. We have the first statement of the lemma by the classification of totally geodesic subsets combined with the fact that there are no geodesic-preserving bijections between $\E^2$ and $\Hyp^2$ due to the parallel postulate being satisfied by one but not the other. Then the second statement follows by applying the classification of geodesic-preserving bijections for $\E^2$ and $\Hyp^2$ due to Jeffers \cite{Jeffers}.
\end{proof}

\begin{lemma} \label{preserved}
         A geodesic-preserving bijection $f: \HxR \rightarrow \HxR$ preserves the classes of geodesics.
\end{lemma}

\begin{proof}

It follows from Lemma~\ref{planes} that $f$ must take horizontal geodesics to horizontal geodesics.

Observe any vertical geodesic $v$ is the intersection of two vertical planes, and conversely, the intersection of two distinct vertical planes is either empty or a vertical geodesic. Since $f$ maps these vertical planes to vertical planes by Lemma~\ref{planes}, we must have that $v$ maps to a vertical geodesic.
\end{proof}

Recall that $f$ acts as an isometry between horizontal planes, and we want to compare these isometries with each other. From the action of $f$ on the vertical geodesics given by Lemma~\ref{preserved}, we get a well-defined bijection $f_\star$ on $\Hyp^2$. Since $f$ sends vertical planes to vertical planes, $f_\star$ takes geodesics to geodesics, and so $f_\star$ is an isometry by \cite{Jeffers}.

\begin{lemma} \label{isomsagree}
 A geodesic-preserving bijection $f: \HxR \rightarrow \HxR$ acts via the same isometry on each horizontal plane.
\end{lemma}

Similarly, since horizontal planes are preserved by Lemma~\ref{preserved}, we define a bijection $f_\R: \R \to \R$ describing the action of $f$ on horizontal planes. 

\begin{lemma} \label{ractionaffine}
        $f_\R$ is an affine transformation of $\R$. 
\end{lemma}

\begin{proof}
 Let $V$ be some vertical plane. By Lemma~\ref{preserved}, $f$ takes $V$ to another vertical plane via an affine map. Since $f$ takes horizontal lines of $V$ to horizontal lines of $f(V)$, and this action on horizontal lines agrees with $f_\R$, it follows $f_\R$ is affine.
\end{proof}

We have the main theorem for this section by combining the above lemmas. 

\begin{proof} [Proof of Theorem~\ref{hyp}]
    Combine Lemma~\ref{isomsagree} and Lemma~\ref{ractionaffine}.
\end{proof}

\section{\texorpdfstring{Geodesic-preserving bijections of $\SxR$}{Geodesic-preserving bijections of S2 times R}} \label{sec: sphere}

Now we show that a geodesic-preserving bijection of $\SxR$ is an isometry composed with an affine map in the $\R$-coordinate. The isometry group is exactly $\Isom(\Sp^2) \times \Isom(\R)$, so we may describe the group of geodesic-preserving bijections as $ \Isom(\Sp^2) \times \Aff(\R) $.

\p{Geodesics} Similar to the previous case, the geodesics of $\Sp^2 \times \R$ are either horizontal, vertical, or slant. Unlike the previous case, however, there is a unique geodesic between two points if and only if they are in a common sphere $\Sp^2 \times \{r\}$. Otherwise, there are infinitely many slant geodesics that spiral around a great circle (a geodesic of the sphere) a different number of times on their journey between the two points.

The totally geodesic subsurfaces of $\Sp^2 \times \R$ are similar to the hyperbolic case -- they are exactly the \textit{horizontal spheres} $\Sp^2 \times \{r\}$ and the \textit{vertical cylinders} $\gamma \times \R$ for a great circle $\gamma$. Proposition~\ref{prop: fundlemma} does not apply to $\Sp^2 \times \R$ since a sequence of slant geodesics through a basepoint can return arbitrarily close to the basepoint, but we do  not appeal to a classification of totally geodesic subsets in this case. We instead achieve preservation of spheres and cylinders through a direct argument using the different possible intersection patterns between the types of geodesics. Then we appeal to the classification of geodesic-preserving bijections of the cylinder which we now discuss.

\subsection{Geodesic-preserving bijections of the cylinder} \label{sec:cylinder}

  We first discuss a new type of geodesic-preserving bijection in the case of the cylinder, $\Sp^1 \times \R$. Parameterizing the $\Sp^1$ coordinate as $\R \mod \Z$, we define a twisting map $t_\alpha$ for some given $\alpha \in \R$ by the following: \[t_\alpha: \Sp^1 \times \R \rightarrow \Sp^1 \times \R\]
 \[(r_1, r_2) \mapsto (r_1 + \alpha r_2, r_2)\]

Note that a nontrivial twisting map is a geodesic-preserving bijection of the cylinder that is not an isometry. Another possibility is a map that fixes the circle coordinate but acts as an affine map in the $\R$-coordinate. We show that the new twisting maps along with the other standard possibilities generate all geodesic-preserving bijections.

\begin{theorem} \label{prop: cylinder2}
    Let $f: \Sp^1 \times \R \rightarrow \Sp^1 \times \R$ be a geodesic-preserving bijection. Then $f$ is a product of 

\begin{itemize}
    \item an isometry,
    \item an affine map in the $\R$-coordinate,
    \item and a twisting map.
\end{itemize}
\end{theorem} 

\p{Outline} Our plan is to show that a geodesic-preserving bijection $f$ descends to a well-defined bijection $f_\star$ on the torus $\Sp^1 \times \Sp^1$ that preserves geodesics, so we can apply a theorem of Limbeek and Shulkin \cite{ShulkinVanLimbeek2017} that says $f_\star$ is an affine map of the torus. We then argue that after composing $f$ with additional geodesic-preserving bijections that $f_\star$ is the identity. It then requires more work to show that $f$ is the identity, since $f_\star$ being the identity only tells us that $f$ is the identity modulo the action of $\Z$ by translations.

We say the \textit{slope} of a geodesic $\gamma$ is $\infty$ when it is vertical, 0 when it is horizontal, and $r$ when $\gamma$ is slant where $r$ is the vertical displacement between two adjacent intersection points of $\gamma$ with a vertical geodesic. In order to get a well-defined sign for the slope we choose that the vertical displacement is measured traveling clockwise around the cylinder along the slant geodesic. 

Our first observation about classes of geodesics under a geodesic-preserving bijection is the following.

\begin{lemma} \label{spherehoriz}
    A geodesic-preserving bijection $f: \Sp^1 \times \R \rightarrow \Sp^1 \times \R$ takes horizontal geodesics to horizontal geodesics.
\end{lemma}

\begin{proof}
    Suppose otherwise, so that either $f$ or $f^{-1}$ takes a vertical or slant geodesic $\gamma$ to a horizontal geodesic. Given two distinct points on $\gamma$, there are infinitely many geodesics between these two points, but $f$ (without loss of generality) maps these points to points with a unique horizontal geodesic between them, a contradiction. \end{proof}

With this in mind, we assume the following simplifications for the lemmas of this section.

\begin{remark} \label{rmk}
For the remainder of the section, we fix a vertical geodesic $v$ and refer to points on $v$ by their height. For example, $0$ refers to the point of height $0$ on $v$. Since $f$ takes horizontal geodesics to horizontal geodesics, the points 0 and 1 are mapped by $f$ to points at different heights. Therefore, after composing $f$ with a twisting map and an affine map in the $\R$-coordinate, we may assume that $f$ fixes 0 and 1. By composing with another twisting map, we may assume $f$ fixes $v$ setwise as well. Our goal is to show that $f$ is the identity after possibly composing with an additional isometry, in particular, a reflection that fixes $v$.
\end{remark}

\begin{lemma}
   $f$ takes vertical geodesics to vertical geodesic and slant geodesics to slant geodesics.
\end{lemma}

\begin{proof}
    Vertical geodesics are preserved because they are exactly the geodesics disjoint from $v$, and $v$ is fixed. Horizontal geodesics are preserved by Lemma~\ref{spherehoriz}, so slant geodesics are preserved too.
\end{proof}

We introduce the notion of \textit{guaranteed sets} and use them to prove some facts about a geodesic-preserving $f$.

\p{Guaranteed sets} For a set of points $P$, we define the guaranteed set $G(P)$ to be the intersection of all the geodesics $\gamma$ such that $P \subset \gamma$. If no such geodesic exists, then we let $G(P) = P$.  When $P$ is a finite set such as $\{x,y,z\}$, we instead denote the guaranteed set $G(x,y,z)$.

Note that guaranteed sets are preserved by geodesic-preserving bijections, meaning $G(f(P)) = f(G(P))$. Also note that $P \subseteq G(P)$, and that $G(P)$ may be a single geodesic. As a basic example, the guaranteed set $G(x)$ for any point $x \in \Sp^2$ is $\{x, \alpha(x)\}$ where $\alpha$ is the antipodal map.

For the cylinder, we use the guaranteed sets $G(a,b)$ from two given points $a,b$. If $a$ and $b$ are at the same height, then $G(a,b)$ is the horizontal geodesic between them. If $a$ and $b$ lie on the same vertical geodesic $v$, then \[G(a,b) = \{a+n(b-a) \mid n \in \Z\} \subset v\] a set of evenly spaced points distance $|b-a|$ apart. Similarly, if $a$ and $b$ are any points at different heights, they are connected by a geodesic $s$ of maximum absolute value slope, and $G(a,b) \subset s$ is a subset of evenly spaced points along $s$.

\begin{lemma} \label{vfixed}
  $f$ fixes any point in $v$ of height $\frac{n}{2^m}$ where $n$ and $m$ are integers.
\end{lemma}

\begin{proof} Since 0 and 1 are fixed by $f$ and $G(0,1) = \Z$, it follows $f$ fixes $\Z$ setwise.
Next, we claim $f(2) = 2$. Since $G(1,2) = \Z$, applying $f$ we have $G(1,f(2)) = f(G(1,2)) = \Z$, but $G(1, x) = \Z$ only if $x$ is either 0 or 2. Since 0 is fixed, we must have $f(2) = 2$. By an induction argument, $f$ then fixes every point with integer height. Now since $1 \in G(0, \frac{1}{2})$, applying $f$ shows that $1 \in G(0, f(\frac{1}{2}))$ and thus 
\[f(\frac{1}{2}) \in \{ \pm \frac{1}{2}, \pm \frac{1}{3}, \cdots \}\] Similarly, since $0 \in G(1, f(\frac{1}{2}))$ we have that
\[f(\frac{1}{2}) \in \{ 1 \pm \frac{1}{2}, 1 \pm \frac{1}{3}, \cdots \}\] 

The only option is that $f(\frac{1}{2}) = \frac{1}{2}$, and it follows from the previous methods that every point of height $\frac{n}{2}$ on $v$ for some integer $n$ is fixed. Repeating this method shows that every point on $v$ of height $\frac{n}{2^m}$ is fixed for integers $m$ and $n$. \end{proof}

\begin{lemma} \label{}
    If $x$ and $y$ lie on the same vertical and $d(x,y)~=~1$, then we have $d(f(x),f(y))~=~1$.
\end{lemma}

\begin{proof}
    We first claim the set of lines of slope $\pm1$ are preserved. Let $s$ be the line of slope 1 through the point 0 on $v$, and let $\Tilde{s}$ be the line of slope -1 through 0. Note that the intersection of these geodesics with $v$ is $\Z$, and these are the only such geodesics with this property. Since $f$ fixes $\Z$ pointwise by Lemma~\ref{vfixed}, it follows $f$ permutes $\{s,\Tilde{s}\}$. Suppose $f(s) = s$, then lines of slope 1 are preserved since they are exactly the geodesics disjoint from $s$. Otherwise if $f(s) = \Tilde{s}$, then similarly the lines of slope 1 are taken to lines of slope -1, and vice versa. 

    Let $l_1$ be the vertical line through $x$ and $y$ and $l_2$ the line of slope 1 through $x$ and $y$. Since $f(l_1)$ is vertical and $f(l_2)$ has slope $\pm 1$, we see that $f(x)$ and $f(y)$ lie on the same vertical and $d(f(x),f(y)) \in \Z$. Since $G(x,y) = l_1 \cap l_2$ and \[G(f(x),f(y)) = f(G(x,y)) = f(l_1) \cap f(l_2)\] corresponds to the points at integer distance from $f(x)$ along $f(l_1)$, it follows  $d(f(x),f(y))=1$.
\end{proof}

The above lemma implies that $f$ descends to a well-defined bijection on the torus $f_\star$. Since $f$ is geodesic-preserving, $f_\star$ is geodesic-preserving as well. Next, we use the following.

\begin{lemma} (Limbeek--Shulkin \cite{ShulkinVanLimbeek2017})
    A geodesic-preserving bijection $f$ on the torus is an affine map.
\end{lemma}

Since $f_\star$ is an affine map of the torus, and it fixes the horizontal and vertical geodesics of the torus (defined as the projections of horizontal and vertical geodesics of the cylinder), we must have that $f_\star$ is a translation of the torus possibly composed with a reflection. Furthermore, since $f_\star$ fixes the projection of $v$ setwise and a point on $v$, it follows $f_\star$ is the identity or a reflection. We thus have the following.

\begin{lemma} \label{almostfixed}
After possibly composing $f$ with the reflection that pointwise fixes $v$, $f$ fixes the cylinder pointwise up to the action of $\Z$ on the $\R$ factor. 
\end{lemma}

Since horizontal geodesics are taken to horizontal geodesics by Lemma~\ref{spherehoriz}, and we are assuming $f$ fixes some points at height $\frac{n}{2^m}$ for integer $n,m$ by Lemma~\ref{vfixed}, Lemma~\ref{almostfixed} implies that any point of height $\frac{n}{2^m}$ is fixed. Thus we have the following

\begin{lemma} \label{geofixed}
  After possibly composing $f$ with the reflection that pointwise fixes $v$, $f$ setwise fixes every geodesic. 
\end{lemma}

\begin{proof} Let $l$ be a geodesic of slope $\alpha \ne 0$ since horizontal geodesics are already fixed. Since $l$ intersects points of heights $\frac{n}{2^m}$ for all $n$ and $m$ and these points are pointwise fixed by Lemmas ~\ref{vfixed} and~\ref{almostfixed}, we are done since $l$ is the only geodesic that intersects all of these points.
\end{proof}

We now give the proof of the main theorem for this section.

\begin{proof}[Proof of Theorem~\ref{prop: cylinder2}]
Let $f$ be any geodesic-preserving bijection. After possibly composing $f$ with additional geodesic-preserving bijections, we have the assumptions of Lemma~\ref{geofixed}. Let $l_\alpha$ be the line of slope $\alpha$ through an arbitrary point $x$. Since $f(x) \in f(l_\alpha) = l_\alpha$ for all $\alpha$ by Lemma~\ref{geofixed}, we must have that $f(x) = x$. 
\end{proof}

\subsection{\texorpdfstring{Proof of the $\SxR$ case}{Proof of the S2 times R case}} \label{sphereproof}
Now we show that a geodesic-preserving bijection of $\SxR$ is an isometry composed with an affine map in the $\R$-coordinate. Note there are no twisting maps in this case -- a similarly defined map that rotates the sphere fibers would send a slant geodesic to a curve whose projection is a small circle (a circle of the sphere that is not a geodesic). 

\begin{theorem}\label{sphere}
    Let $f: \Sp^2 \times \R \to \Sp^2 \times \R$ be geodesic-preserving bijection, then $f$ is an isometry of $\Sp^2 \times \R$ composed with an affine transformation on the real line component.
\end{theorem} 

We first show that a geodesic-preserving bijection preserves horizontal, vertical, and slant geodesics.

\begin{lemma}
    Let $f: \SxR \rightarrow \SxR$ be a geodesic-preserving bijection. Then $f$ preserves classes of geodesics.
\end{lemma}

\begin{proof}
    Firstly, $f$ preserves horizontal geodesics since they are the only geodesics that can intersect another geodesic exactly twice.
    
    Now observe that a slant geodesic and a vertical geodesic intersect infinitely often or are disjoint. However, two distinct slant geodesics may intersect exactly once by choosing them to start at the same point and then giving one rational and the other irrational slope. Here slope refers to the vertical displacement the slant geodesic makes after tracing out a great circle once in the projection. The lemma now follows.
\end{proof}

Now we combine this with the cylinder result for the final proof.

\begin{proof} [Proof of Theorem \ref{sphere}]
    Since $f$ preserves vertical geodesics, the action of $f$ on the verticals determines a bijection $f_\star: \Sp^2 \rightarrow \Sp^2$, and $f_\star$ is geodesic-preserving since horizontals are preserved. Since $f$ preserves horizontal geodesics, it also preserves horizontal spheres. By the classification of geodesic-preserving bijections for spheres \cite{Jeffers}, we have that $f_\star$ is an isometry composed with a bijection swapping some set of points with their antipodal points. Composing $f$ with an isometry we can force $f$ to setwise fix a chosen sphere $S$, and additionally we can assume $f_\star$ fixes all geodesics of $S$ setwise. 
    
  Using that $f$ preserves verticals, note every cylinder is setwise fixed by $f$. Now $f$ restricts to a geodesic-preserving bijection on each cylinder, so we can appeal to the classification of these bijections in Theorem~\ref{prop: cylinder2}. The twisting maps take vertical to slant geodesics in the cylinder, but we have that verticals are preserved, so $f$ restricts to an isometry composed with an affine map in the $\R$-coordinate. 

    Composing $f$ with another isometry if needed we have $f(x) = x$ for some $x \in S$. Now the vertical line through $x$ is fixed, so $f$ restricted to any cylinder containing $x$ is an affine map in the $\R$-coordinate. Since spheres are preserved, this affine map is the same for every cylinder, so $f$ acts as an affine map in the $\R$-coordinate on the entirety of $\SxR$. We are now done. \end{proof}

\section{\texorpdfstring{Geodesic-preserving bijections of $\sltilde$}{Geodesic-preserving bijections of the universal cover of SL2R}} \label{sec: sltilde}

 Recall $\sltilde$ is the universal cover of the Lie group $\slg$. We equip $\slg$ with the left-invariant Riemannian metric with respect to the Lie group structure, and we equip $\sltilde$ with the pullback metric. Our goal in this section is to prove the following.

\begin{theorem} \label{sl}
    Let $f: \sltilde \rightarrow \sltilde$ be a geodesic-preserving bijection. Then $f$ is an isometry.
\end{theorem}

 \subsection{Background}  We first give some background on the geometry, its geodesics, and its isometry group mostly following the manuscript of Scott \cite{Scott}.
 
 \p{Sasaki metric} It is known that the left-invariant Lie group metric on $\pslg$ is isometric to the Sasaki metric on $\ut$, the unit tangent bundle of $\Hyp^2$. The Sasaki metric is the natural metric such that the projection to $\Hyp^2$ is a Riemannian submersion, and the metric restricted to a fiber over a point is the standard Euclidean metric on the circle. Since $\slg$ double covers $\pslg$, the universal cover of $\pslg$ is also $\sltilde$. Thus, the pullback of the Sasaki metric gives a metric on $\uttilde$ that is isometric to the metric on $\sltilde$. 
 
 Throughout the section, we  identify $\sltilde$ with $\uttilde$. One can think of points in $\sltilde$ as based unit vectors with an additional winding number measuring how many times a vector has made a full rotation.

\p{Lifting from $\Hyp^2$ with parallel transport} Let $p: \sltilde \rightarrow \Hyp^2$ denote the projection map. There are natural maps $t_x: \Hyp^2 \rightarrow \sltilde$ for a given basepoint $x \in \sltilde$ given by parallel transport of unit vectors. Towards defining these maps, let $y \in \Hyp^2$ and let $\gamma$ be the geodesic from $p(x)$ to $y$. We then define $t_x(y)$ as the unit vector we arrive at after parallel transport of $x$ along $\gamma$ to $y$. We refer to these embeddings of $\Hyp^2$ as \textit{horizontal planes}, and we call a geodesic \textit{horizontal} when it lies inside a horizontal plane. Note any geodesic $\gamma$ of $\Hyp^2$ along with a given $x$ in $\sltilde$ with projection in $\gamma$ determines a unique horizontal geodesic.

One important warning is that the horizontal planes are not totally geodesic. Consider a geodesic triangle in $\Hyp^2$, and consider the geodesic path that forms a loop traveling around the triangle. We can lift this path to a path of horizontal geodesic segments in $\sltilde$. Since parallel transport of a unit vector along the geodesic triangle causes the vector to return to the original point in a rotated position, the lifted path is not a loop. The exact amount of rotation is called the \textit{holonomy} of the triangle, and it is known to be the area of the triangle ($\pi$ minus the sum of the angles).

 \p{Geodesics} We describe the geodesics of $\sltilde$ using a description of the geodesics of the Sasaki metric from \cite{Nagy1977}. One initial guess is that the geodesics are exactly translates of one-parameter subgroups in the Lie group sense. This is known to hold for Lie groups admitting a \textit{bi-invariant} Riemannian metric; however, the metric on $\sltilde$ is known to be left-invariant, but not right-invariant. We group the geodesics into three classes:  
 
\begin{itemize}
    \item horizontal geodesics contained in a horizontal plane. These project to a geodesic in $\Hyp^2$.
    \item vertical geodesics such that the projection to $\Hyp^2$ is a point. These geodesics are the fibers of the line bundle structure.
    \item slant geodesics such that the projection to $\Hyp^2$ is a curve of non-zero constant curvature so either
    \begin{itemize}
        \item a circle,
        \item a horocycle,
        \item or, a hypercycle.
    \end{itemize}
\end{itemize}

The curves of non-zero constant curvature $K$ in $\Hyp^2$ are exactly circles ($K >1$), horocycles ($K=1$), and the hypercycles ($K<1)$. A slant geodesic given a unit-speed parameterization traces out the curve in the projection at unit speed, and also has a constant vertical speed depending linearly on the curvature of the projection curve. 

Proposition~\ref{prop: fundlemma} applies to $\sltilde$ from the above description of its geodesics. It is known that there are no nontrivial totally geodesic submanifolds in $\sltilde$ (see \cite[Theorem 7.2]{Tsukada1996}), so we have that there are no nontrivial totally geodesic subsets.

\p{Isometries} Note $\sltilde$ has a natural bundle structure over $\Hyp^2$ where the fibers are exactly the vertical geodesics. The isometries of $\sltilde$ preserve this bundle structure, and the group of isometries fits into the following short exact sequence.
\[
\begin{tikzcd}
1 \arrow[r] & \R \arrow[r] & \Isom(\sltilde) \arrow[r] & \Isom(\Hyp^2) \arrow[r] & 1
\end{tikzcd}
\]

The action of $\R$ is via ``winding maps" that rotate each unit vector representing a point $\sltilde$ by a chosen amount. There is also an action of $\Isom(\Hyp^2)$ on $\Isom(\sltilde)$ given by the natural action of isometries on unit vectors, though the exact sequence does not split since a rotation from 0 to $2\pi$ is a loop that lifts to a path of winding maps from 0 to $2\pi$, and the $2\pi$ winding map is nontrivial in $\Isom(\sltilde)$. 

The stabilizer of a point is known to be the orthogonal group $O(2)$, so we can refer to rotations in $\sltilde$ about a given point. These rotations can be thought of as a composition of an isometry induced by a rotation of $\Hyp^2$ with a winding map that reverses the winding on the fixed point caused by the rotation.

There are two components for $\Isom(\sltilde)$ corresponding to orientation-reversing and orientation-preserving isometries of $\Hyp^2$. Interestingly, an isometry of $\sltilde$ induced by an orientation-reversing isometry of $\Hyp^2$ reverses the orientation of the fibers, so every isometry of $\sltilde$ is orientation-preserving.

\subsection{Constant curvature curves} One of our main tools in this case is Proposition~\ref{constantcurvature} involving bijections that preserve constant curvature curves of $\Hyp^2$. The proof uses a simple observation essentially about tangent lines to a smooth immersed curve in $\R^2$ -- if we draw a tangent line to uncountably many points in the curve, then some tangent lines must intersect. More precisely,

\begin{proposition} \label{tangent}
    Suppose $\gamma$ is a smooth immersed curve in $\R^2$, and $\{\gamma_x\}_{x \in I}$ is a family of smooth paths for some uncountable $I \subset \gamma$ such that $\gamma_x \cap \gamma = \{x\}$, $\gamma_x$ does not cross sides of $\gamma$ at $x$, and the intersection occurs in the interior of $\gamma_x$. Then some distinct $\gamma_x$ intersect.
\end{proposition}

We give some examples to explain the additional hypotheses on the $\gamma_x$. Consider an embedded curve $\gamma$ and an embedding given by a tubular neighborhood theorem of $I \times I$ into $\R^2$ such that the image of $I \times \{\frac{1}{2}\}$ is a subpath of $\gamma$. The image of curves of the form $y = (x-a)^3$ for varying $a$ gives an uncountable smooth family of pairwise disjoint paths that cross sides of $\gamma$ but are each tangent to $\gamma$. By considering only the positive part of these paths, we see the necessity for assuming the intersection occurs in the interior of the $\gamma_x$. Smoothness of the $\gamma_x$ is also required since otherwise we could use paths that travel down the previous paths and then make a discontinuous turn at $x$ to go back up the same path.

This situation is highly similar to a known result about embedding different objects into $\R^2$. As in the previous paragraph, there exists an uncountable family of pairwise disjoint embeddings of the interval into $\R^2$. On the other hand, let the \textit{tripod} (Moore instead uses the term triod) $Y$ refer to the topological space defined by identifying one endpoint from three intervals. Moore showed the following.

\begin{lemma} \cite{Moore1928Triods} \label{tripod}
Any family of embeddings $Y \rightarrow \R^2$ with pairwise disjoint images is countable.
\end{lemma}

The proof relies on the countability of the rational points in $\R^2$ -- given a pairwise disjoint family of embeddings of $Y$, one needs to show that each can be assigned a unique rational point in $\R^2$, and then it follows that the family is countable. Assuming Lemma~\ref{tripod}, we show the desired claim.

\begin{proof} [Proof of Proposition~\ref{tangent}]
   Assume by contradiction that the paths $\gamma_x$ are pairwise disjoint. Since there are uncountably many paths, uncountably many of the intersections with $\gamma$ occur in a small embedded subarc of $\gamma$ locally on the same side of $\gamma$. We restrict the set of $\gamma_x$ to just the ones intersecting $\gamma$ in this subarc. Consider a family of pairwise disjoint paths $\{\gamma^\prime_x\}$ such that $\gamma \cap \gamma^\prime_x = \{x\}$ and $\gamma^\prime_x$ is orthogonal to $\gamma$ at $x$. By taking the union of a subpath of $\gamma^\prime_x$ with a subpath of $\gamma_x$, we get a tripod $T_x$ embedded in $\R^2$. By choosing the subpaths of $\gamma^\prime_x$ on sides opposite to $\gamma_x$, we then have an uncountable family of pairwise disjoint tripod spaces, a contradiction by Lemma~\ref{tripod}.
\end{proof}

We are now ready to prove the following. A constant curvature curve is assumed to be complete.

\begin{proposition} \label{constantcurvature}
     Let $f: \Hyp^2 \to \Hyp^2$ be a bijection that preserves constant curvature curves. Then $f$ is an isometry.
\end{proposition} 

\begin{proof}
   We will show that $f$ preserves circles, and then the proposition follows from a result of Carathéodory \cite{caratheodory1937}. Carathéodory's result shows a circle-preserving bijection on a domain of the plane is a M\"{o}bius transformation, so in particular it applies to the Poincaré model of $\Hyp^2$, and M\"{o}bius transformations are isometries of $\Hyp^2$.
   
 Note for any non-circle constant curvature curve $\gamma$ and any choice of point $x \in \gamma$ there exists a constant curvature curve $\gamma_x$ such that $\gamma \cap \gamma_x = \{x\}$, and we can choose the $\gamma_x$ to be pairwise disjoint. 

However, if we have a circle $\gamma$ and a family of constant curvature curves $\{\gamma_x\}$ such that $\gamma \cap \gamma_x = \{x\}$, then each $\gamma_x$ is non-crossing tangent to $\gamma$, so some of the $\gamma_x$ must intersect by Proposition~\ref{tangent}. It follows $f$ preserves circles, so we are done.\end{proof}

\subsection{Classification of geodesic-preserving bijections} We start with a similar method to the previous cases of showing classes of geodesics are preserved by our geodesic-preserving bijection. We first show that vertical geodesics are sent to vertical geodesics by $f$. Since each point in $\Hyp^2$ determines a unique vertical geodesic, this allows us to define a bijection $f_\star: \Hyp^2 \rightarrow \Hyp^2$ given by the action of $f$ on vertical geodesics. We then claim $f_\star$ takes constant curvature curves to constant curvature curves. To see this, first note that $f$ permutes the horizontal and slant geodesics all of which project to a constant curvature curve in $\Hyp^2$. When $f$ acts on a geodesic, it acts on all the vertical translates of that geodesic, so the claim follows. 

\begin{lemma} \label{slvert}
    Let $f: \sltilde \to \sltilde$ be a geodesic-preserving bijection. Then $f$ preserves vertical geodesics.
\end{lemma}

\begin{proof}
    We observe that vertical and slant geodesics are precisely the classes of geodesics that may intersect other geodesics infinitely often. Indeed, we find that horizontals, slants with horocycle projection, and slants with hypercycle projection cannot intersect another geodesic infinitely often by considering their projections onto $\Hyp^2$. Since $f$ must preserve intersection types of geodesics, it sends vertical geodesics to verticals or slants. Finally, we observe that two slant geodesics may intersect exactly once, while a vertical cannot intersect a slant exactly once nor can it intersect another vertical. It follows that verticals get sent to verticals. 
\end{proof}

Combining the previous lemmas we have

\begin{lemma}

Let $f: \sltilde \rightarrow \sltilde$ be a geodesic-preserving bijection. Then the bijection $f_\star: \Hyp^2 \rightarrow \Hyp^2$ given by the action on vertical geodesics is an isometry. Furthermore, $f$ preserves the class of all geodesics.
\end{lemma}

\begin{proof}
 By Lemma~\ref{slvert}, $f_\star$ is well-defined. Since $f_\star$ preserves constant curvature curves, we have by Proposition~\ref{constantcurvature} that $f_\star$ is an isometry. Now $f$ preserves the classes of all geodesics since these were defined by the type of constant curvature curve in the projection, and an isometry preserves the curvature.
\end{proof}

We use this with the classification of bijections of $\Hyp^2$ preserving constant curvature curves to prove the desired result. 

\begin{proof} [Proof of Theorem \ref{sl}]

    Choose a point $x \in \sltilde$, and compose $f$ with some isometry and rename the result $f$ so that now $f$ fixes $x$. We can then compose with a rotation about $x$, so that $f_\star$ is the identity. We claim $f$ pointwise fixes each horizontal geodesic $h$ through $x$. For any point $y \in h$ other than $x$, there is a vertical geodesic $v$ such that $v \cap h = \{y\}$. Since $f(v)=v$, and $f(h)$ must be a horizontal geodesic intersecting $f(v)$ and $x$, it follows $f$ fixes $h$. The claim then follows since the intersection of $v$ and $h$ must now be fixed. We then claim any two points of $\sltilde$ are connected by a path of finitely many horizontal geodesic segments, so an induction argument shows $f$ pointwise fixes all of $\sltilde$.

    To see the final claim, note that we can assume the points project to the same point $\Hyp^2$ by following from one point along a horizontal path. Then we choose some closed polygonal path in $\Hyp^2$ of area equal to the vertical distance between the points, and lift each side of the polygon to a horizontal geodesic segment.
 \end{proof}

\section{Geodesic-preserving bijections of Nil} \label{sec: nil}

The $\Nil$ geometry is given by the Lie group structure on the Heisenberg group. It has the following group operation on $\R^3$ and left-invariant metric: \[(x,y,z) \cdot (a,b,c) = (x+a,y+b,z+c+xb)\]
 \[ds^2 = dx^2 + dy^2 + (dz-xdy)^2\] Our goal in this section is to prove the following. 

\begin{theorem} \label{nil}
    Let $f: \Nil \to \Nil$ be a geodesic-preserving bijection. Then $f$ is an isometry.
\end{theorem}

\subsection{Background}

We first give some background again following Scott \cite{Scott}. $\Nil$ has the metric structure of a line bundle over $\E^2$. We parametrize $\Nil$ so the $xy$-plane is this Euclidean base space. As with $\sltilde$, this horizontal plane is not a totally geodesic subsurface, but the restriction of the metric to the plane gives $\E^2$. 

\p{Geodesics} Since the left-invariant metric on $\Nil$ is not also right-invariant, some geodesics are not translates of one-parameter subgroups. The geodesics of $\Nil$ are explicitly computed in the solution to Exercise 2.90 bis (c) in \cite[Appendix~B]{gallot}. Following this description, the geodesics in $\Nil$ can be split into the following classes:

\begin{itemize}
    \item parabolic geodesics such that the projection to $\E^2$ are lines. The actual shape of these geodesics are either actual horizontal lines or parabolas.
    \item vertical geodesics such that the projection to $\E^2$ is a point. These geodesics are the fibers of the line bundle structure.
    \item slant geodesics such that the projection to $\E^2$ is a circle.

\end{itemize}

The slant geodesics are ``spirals" whose projection traces out a circle infinitely many times. In the given parameterization of $\Nil$, the exact vertical speed of a slant geodesic at any moment is difficult to describe qualitatively, but the spiral makes consistent vertical progress each time it traces out a circle.

From this description of the geodesics of $\Nil$, we have that Proposition~\ref{prop: fundlemma} applies. It is also known $\Nil$ has no nontrivial totally geodesic submanifolds (see \cite[Theorem 7.2]{Tsukada1996}), and thus we have that there are no nontrivial totally geodesic subsets.

\p{Isometries} For $\Nil$, all isometries are known to preserve the line bundle structure, thus descending to an isometry on $\E^2$. The isometry group of $\Nil$ fits into the short exact sequence 
\[
\begin{tikzcd}
1 \arrow[r] & \R \arrow[r] & \Isom(\Nil) \arrow[r] & \Isom(\E^2) \arrow[r] & 1
\end{tikzcd}
\]

Since $\Isom(\E^2)$ has two components corresponding to orientation-preserving and orientation-reversing isometries, $\Isom(\Nil)$ also has two components. Unlike the $\sltilde$ case, $\Nil$ cannot be thought of the universal cover of the unit vector space of $\E^2$ (which is $\E^3$), so we do not have a natural action of $\Isom(\E^2)$ on $\Nil$. Similar to the $\sltilde$ case though, every isometry of $\Nil$ is orientation-preserving. 

Another description for $\Isom(\Nil)$ is $\Nil \rtimes O(2)$ where $O(2)$ is the orthogonal group. Here $O(2)$ is also the identity component of the stabilizer of a point, so we can refer to rotations about a given point.

\subsection{Classification of geodesic-preserving bijections} Due to the similarity in structure to the $\sltilde$ case, we  use a very similar argument for the classification of geodesic-preserving maps. 

\begin{lemma} \label{lemma:nilclass}
    Let $f: \Nil \to \Nil$ be a geodesic-preserving bijection. Then $f$ preserves classes of geodesics.
\end{lemma}

\begin{proof}
    We observe that vertical and slant geodesics are precisely the classes of geodesics that may intersect other geodesics infinitely often. Indeed, we find that parabolics cannot intersect another geodesic infinitely often by considering their projections onto $\E^2$. Since $f$ must preserve intersection types of geodesics, it sends vertical geodesics to verticals or slants. Finally, we claim that two slant geodesics may intersect finitely often, while a vertical cannot intersect a slant finitely often nor can it intersect another vertical. It follows that verticals get sent to verticals. 

    The final claim is harder than in the $\sltilde$ case, but we can check the exact parametric equations for the slant geodesics (see the solution to Exercise 2.90 bis (c)  in \cite[Appendix~B]{gallot}), and see it is possible to choose slants starting from the same point where one has rational and the other has irrational displacement each time it hits the vertical geodesic over the initial point.
\end{proof}

Since vertical geodesics are now preserved by Lemma~\ref{lemma:nilclass}, we have a well-defined bijection $f_\star: \E^2 \rightarrow \E^2$ given by the action of $f$ on the vertical geodesics. Next, we show $f_\star$ is an isometry.

\begin{lemma}
     Suppose $f: \Nil \rightarrow \Nil$ is a geodesic-preserving bijection, then the induced bijection $f_\star:\E^2 \rightarrow \E^2$ given by the action of $f$ on vertical geodesics is an isometry.
 \end{lemma}

\begin{proof}
 Note $f_\star$ preserves lines since parabolics are preserved. We have that $f_\star$ is an affine map by Jeffers \cite{Jeffers}, but possibly not an isometry. However, $f_\star$ preserves circles since slant geodesics are preserved, so $f_\star$ must be an isometry.
\end{proof}

 We are now ready to prove the desired result.
 
\begin{proof} [Proof of Theorem \ref{nil}]
Given the above lemma, we can compose $f$ with an isometry of $\Nil$ and rename the result $f$ so that the bijection $f_\star:\E^2 \rightarrow \E^2$ induced by the action on vertical lines is the identity, i.e., $f$ fixes each vertical line setwise. Furthermore, composing with a vertical translation we can assume $f$ fixes the origin. Since $f$ also preserves parabolics, it follows that $f$ setwise fixes each parabolic geodesic through the origin since there is a unique parabolic geodesic intersecting the origin and that vertical line. Furthermore, $f$ pointwise fixes each parabolic through the origin since each point is the intersection with a vertical geodesic setwise fixed by $f$.

We are now done by induction since arbitrary distinct points are connected by a finite path of parabolic geodesics. For this last point, we use the known fact that one can get from a given point to any point on the vertical geodesic through it by traveling along the parabolic lifts of the sides of an appropriately chosen square.
\end{proof}

\section{Geodesic-preserving bijections of Sol} \label{sec: sol}

Now we discuss perhaps the most complicated Thurston geometry, $\Sol$. However, due to an abundance of totally geodesic subsurfaces, the classification of geodesic-preserving bijections is simpler than the other cases. 

\begin{theorem} \label{sol}
      A geodesic-preserving bijection $f: \Sol \rightarrow \Sol$ is an isometry.
\end{theorem}

$\Sol$ is a Lie group with the following group operation on $\R^3$ and left-invariant metric: 
\[(x, y, z) \cdot (a, b, c) = (e^{-z}a + x, e^{z}b + y, c + z)\]
\[ds^2 = e^{2z}dx^2 + e^{-2z}dy^2 + dz^2\]

\subsection{Background}

Sol arises from the semidirect product $\R^2 \rtimes \R$ where $z\in\R$ acts by $(x, y) \to (e^zx, e^{-z}y)$. Thus we can consider $\Sol$ as a plane bundle over $\R$. Each plane in this bundle corresponds to setting $z$ to a constant and is isometric to the Euclidean plane. However, these planes are known to be not totally geodesic. 

By setting $x$ or $y$ to a constant in the metric equation above, we get a metric on $\R^2$ that is isometric to the Poincar\'e metric on the upper half plane. One can check setting $x$ to a constant and then letting $u= y$ and $v = e^{z}$ gives the metric $\frac{du^2 + dv^2}{v^2}$ and sends $\R^2$ to the upper half plane. When $y$ is constant, then we use $u = x$ and $v = e^{-z}$ instead. In fact, these hyperbolic planes given by setting $x$ or $y$ are known to be exactly the totally geodesic subsurfaces of $\Sol$.

\p{Geodesics} The geodesics of $\Sol$ contained in a totally geodesic plane are easy to understand since each totally geodesic plane is essentially a stretched out version of the Poincar\'e model -- vertical lines are geodesics, and every other geodesic tends towards infinity in either the positive or negative $z$ direction depending on whether the plane has $x$ or $y$ constant. The generic geodesics of $\Sol$ are complicated, so we do not make use of them beyond knowing that they do not violate Proposition~\ref{prop: fundlemma}. In fact, a generic geodesic spirals around a ``cylinder" and moves monotonically along its axis, so the geodesic never returns close to a given basepoint (for more details, see the thesis of Grayson \cite{Grayson1983}).

From the above description of the geodesics, we have that Proposition~\ref{prop: fundlemma} holds for $\Sol$, so together with the known classification of totally geodesic subsurfaces (see \cite[Theorem 7.2]{Tsukada1996}), we have the following.

\begin{classsol}
The nontrivial totally geodesic subsets of $\Sol$ are exactly the hyperbolic planes where $x$ or $y$ is constant.
\end{classsol}

\p{Isometries} We give some background on the isometry group following Scott \cite{Scott}. The identity component of isometries of $\Sol$ is $\Sol$ itself, but there are eight components for the isometry group. Isometries in the other components come from composing an isometry in the identity component with an isometry fixing a point.

A point stabilizer group is isomorphic to the dihedral group of order eight, $D_4$. The isometries fixing the origin are the maps $(x, y, z) \mapsto (\pm x, \pm y, z)$ and $(x, y, z) \mapsto (\pm y , \pm x, -z)$. Note the first type of map fixes each hyperbolic plane, and the second type of map swaps the hyperbolic planes with the $z$ coordinate negated since the hyperbolic geodesics tend to negative infinity in one plane and positive infinity in the other.

\subsection{Classification of geodesic-preserving bijections} As a consequence of the classification of totally geodesic subsets, we have the following.

\begin{lemma} \label{solvert}
    A geodesic-preserving bijection $f: \Sol \rightarrow \Sol$ takes vertical lines to vertical lines.
\end{lemma}

\begin{proof}
    Observe that each vertical line is the intersection of two orthogonal hyperbolic planes. Since these planes are exactly the totally geodesic subsets of $\Sol$, orthogonal hyperbolic planes are mapped by $f$ to orthogonal hyperbolic planes, and so vertical lines must be mapped to vertical lines.
\end{proof}

With Lemma \ref{solvert} and the Jeffers result classifying geodesic-preserving bijections on hyperbolic spaces \cite{Jeffers}, we give the proof of the Sol case.

\begin{proof} [Proof of Theorem \ref{sol}]
  Since the hyperbolic planes are the only nontrivial totally geodesic subsets, $f$ preserves the hyperbolic planes. By composing $f$ with an isometry of $\Sol$, we can assume $f$ fixes the origin and setwise fixes the hyperbolic planes, $H_1$ and $H_2$, through the origin. By the classification of geodesic-preserving bijections for $\Hyp^2$ in \cite{Jeffers}, $f$ acts as an isometry on each $H_i$. 
  
  Since $H_1 \cap H_2$ is a vertical geodesic $v$ through the origin, $f$ setwise fixes $v$.  Since $f|_{H_i}$ is an isometry fixing $v$ and the origin, we have four possibilities for $f|_{H_i}$ corresponding to the group of isometries generated by the reflection across $v$ and the reflection across the geodesic orthogonal to $v$ at the origin. By Lemma~\ref{solvert}, we must have that $f|_{H_i}$ sends vertical lines to vertical lines, and this is satisfied only by the identity or the reflection about $v$. 

  Now by composing with another isometry, we can assume each $f|_{H_i}$ is the identity. Observe $\Sol$ is covered by planes $\{O_\alpha\}$ orthogonal to $H_1$ (one of them being $H_2$), each intersecting it in a vertical geodesic $v_\alpha$, and so $f$ setwise fixes each $O_\alpha$. Thus, $f$ restricts to an isometry on each, and since the vertical geodesic is pointwise fixed, we have $f$ restricts to either the identity or the reflection about $v_\alpha$. 

   Finally, $f$ restricts to the same map on each $O_\alpha$, since otherwise any hyperbolic plane parallel to $H_1$ is not mapped to another hyperbolic plane. Since $f$  is the identity on $H_2$, we have that $f$ is globally the identity.
   \end{proof}
\bibliography{bib}
\bibliographystyle{plain}
\end{document}